\documentclass[journal]{IEEEtran}
 \usepackage[cmex10]{amsmath}
\usepackage{amssymb}
 \usepackage{amsthm}
\usepackage{graphicx}    
\usepackage{enumerate}
\usepackage{array}

\usepackage{mdwmath}
\usepackage{mdwtab}
 \usepackage{amsfonts}    
\usepackage{url}          

\newtheorem{thm}{\bf{Theorem}}
\newtheorem{prop}{\bf{Proposition}}

\newlength{\halfpagewidth}
    \divide\halfpagewidth by 2

\begin{document}


\title{Boundary Control of Coupled Reaction-Advection-Diffusion Systems with Spatially-Varying Coefficients} 


\author{Rafael~Vazquez
~and~Miroslav~Krstic~%
\thanks{R. Vazquez is with the Department of Aerospace Engineering, University of Seville, Seville, 41092, Spain (e-mail: rvazquez1@us.es)}
\thanks{M. Krstic is with
the Department of Mechanical Aerospace Engineering,
 University of California, San Diego, CA 92093-0411, USA 
(e-mail: krstic@ucsd.edu)}}


\markboth{IEEE Transactions On Automatic Control}%
{Shell \MakeLowercase{\textit{et al.}}: Bare Demo of IEEEtran.cls for Journals}

\maketitle

\begin{abstract}                          
Recently, the problem of boundary stabilization for unstable linear constant-coefficient coupled reaction-diffusion systems was solved by means of the backstepping method. The extension of this result to systems with advection terms and spatially-varying coefficients is challenging due to complex boundary conditions that appear in the equations verified by the control kernels. In this paper we address this issue by showing that these equations are essentially equivalent to those verified by  the control kernels for first-order hyperbolic coupled systems, which were recently found to be well-posed. The result therefore applies in this case, allowing us to prove $H^1$ stability for the closed-loop system. It also shows an interesting connection between backstepping kernels for coupled parabolic and hyperbolic problems.
\end{abstract}


%
\IEEEpeerreviewmaketitle

\section{Introduction}
\IEEEPARstart{I}{n} a recent work~\cite{orlov}, the problem of boundary stabilization for general  linear \emph{constant-coefficient} coupled reaction-diffusion systems was resolved by means of the backstepping method~\cite{krstic}. However, the extension of this result to systems with \emph{advection} terms and \emph{spatially-varying} coefficients---as usually found in applications---is far from trivial. The main difficulty arises when trying to solve the partial differential equations verified by the control kernels (usually known as the ``backstepping kernel equations''). For $n$ states in a system of coupled parabolic equations, one needs to find $n^2$ control kernel verifying $n^2$ fully coupled \emph{second-order} hyperbolic equations in a triangular domain, with complicated boundary conditions. In the constant-coefficient case, it is possible to simplify the boundary conditions by assuming a certain kernel structure, and then the equations can be readily solved~\cite{orlov}. However, this procedure does not extend to the spatially-varying case and/or advection terms. In this work, we show that the kernel equations can be written (using some non-trivially-defined intermediate kernels) as a coupled system of $2n^2$ \emph{first-order} hyperbolic equations. Interestingly, these kernel equations are very similar to those found when applying backstepping to find boundary controllers for first-order hyperbolic coupled systems~\cite{long-linear}. A result recently obtained for this problem showed that the resulting kernel equations were well-posed and had piecewise differentiable solutions~\cite{long-nonlinear}. Applying this result in our case allows us to find a backstepping controller, and to prove $H^1$ exponential stability for the origin of the closed-loop system with arbitrary convergence rate. Our result shows an interesting and non-trivial connection between backstepping controllers for coupled parabolic and hyperbolic systems.

The problem presented in this paper could be addressed by other methods, including the semigroup approach (see for instance~\cite{liu}), eigenvalue assignment~\cite{barbu}, flatness~\cite{meurer2009tracking} or LQR~\cite{moura}. The main advantage of backstepping is that, once the well-posedness of the kernel equations has been established, analytical and numerical results are simple to obtain, including the well-posedness of the closed loop system in high-order Sobolev spaces or even explicit exact controllers~\cite{Vazquez2014}. Backstepping has proved itself to be an ubiquitous method for PDE control, with many other applications including, among others, flow control~\cite{vazquez,vazquez-coron}, nonlinear PDEs~\cite{vazquez2}, disturbance rejection~\cite{anfinsen2015disturbance,Hasan}, hyperbolic 1-D systems~\cite{vazquez-nonlinear,florent,krstic3}, adaptive control~\cite{krstic4}, wave equations~\cite{krstic2}, Korteweg-de Vries equations~\cite{cerpa}, and delays~\cite{krstic5}. Other recent results related to boundary control of parabolic systems include~\cite{jie}, where backstepping is applied to find multi-agent deployments in 3-D space, output-feedback boundary control for ball-shaped domains in any dimension~\cite{nball}, and design of output feedback laws for convection problems on annular domains (see~\cite{convloop}).

The structure of the paper is as follows. In Section~\ref{sec-plant} we introduce the problem and state our main result. We explain our design method (backstepping) and  show the stability of the closed-loop system in Section~\ref{sec-design}. Next, we prove that there is a solution to the backstepping kernel equations in Section~\ref{sec-kernels}. We conclude the paper with some remarks in Section~\ref{sec-conclusions}.

\section{Coupled reaction-advection-diffusion systems}\label{sec-plant}
Consider the following general linear spatially-varying reaction-advection-diffusion system
\begin{equation} \label{eqn-reactdif}
u_t=\partial_x\left(\Sigma(x) u_{x}\right)+\Phi(x) u_x+\Lambda(x)u,
\end{equation}
for $x\in[0,1]$, $t>0$, with $u\in\mathbb{R}^n$ defined as
\begin{eqnarray}
u&=&\left[\begin{array}{c}u_1\\ u_2 \\ \vdots \\ u_n \end{array}\right],
\end{eqnarray}
and the various coefficients appearing in (\ref{eqn-reactdif}) defined as
\begin{eqnarray}
\Sigma(x)&=&\left[\begin{array}{cccc}\epsilon_1(x)&0&\hdots&0 \\ 0&\epsilon_2(x)&\hdots&0 \\
\vdots&\vdots&\ddots&\vdots\\
 0&0&\hdots & \epsilon_N(x) \end{array}\right], \\ \Lambda(x)&=&\left[\begin{array}{cccc}\lambda_{11}(x)&\lambda_{12}(x)&\hdots&\lambda_{1n}(x) \\ \lambda_{21}(x)&\lambda_{22}(x)&\hdots&\lambda_{2n}(x) \\
\vdots&\vdots&\ddots&\vdots\\
 \lambda_{n1}(x)&\lambda_{n2}(x)&\hdots & \lambda_{nn}(x) \end{array}\right],\\
\Phi(x)&=&\left[\begin{array}{cccc}\phi_{11}(x)&\phi_{12}(x)&\hdots&\phi_{1n}(x) \\ \phi_{21}(x)&\phi_{22}(x)&\hdots&\phi_{2n}(x) \\
\vdots&\vdots&\ddots&\vdots\\
 \phi_{n1}(x)&\phi_{n2}(x)&\hdots & \phi_{nn}(x) \end{array}\right].
\end{eqnarray}
and 
with boundary conditions
\begin{eqnarray}
u(0,t)&=&0,\label{eqn-u0}\\
u(1,t)&=&U(t),\label{eqn-u1}
\end{eqnarray}
where $U(t)$ is the actuation, defined as
\begin{equation}
U(t)=\left[\begin{array}{c}U_1(t)\\ U_2(t) \\ \vdots \\ U_n(t) \end{array}\right].
\end{equation}
The only assumption on (\ref{eqn-reactdif}),(\ref{eqn-u0})--(\ref{eqn-u1}) is that these coefficients are sufficiently regular; in particular, it is required that the entries of $\Sigma(x)$ are three times differentiable, those of $\Phi(x)$ twice differentiable and those of $\Lambda(x)$ differentiable. In addition,
we assume that the states are ordered so that $\bar{\epsilon}\geq\epsilon_1(x)>\epsilon_2(x)>\hdots>\epsilon_n(x)\leq\underline{\epsilon}>0$. The diffusion coefficients could also be equal at some (or all) values of $x$ but to avoid technical complications we confine ourselves to the case of strict inequality.

Since (\ref{eqn-reactdif}), (\ref{eqn-u0})--(\ref{eqn-u1}) is potentially unstable depending on the values of the coefficients, the problem we consider is the design of a (full-state) feedback control law for $U(t)$ that makes the system stable for any possible value of the coefficients. 

We will make use of the $L^2([0,1])$ and $H^1([0,1])$ spaces, defined, respectively, as the space of square-integrable vector functions in the $[0,1]$ interval and the space of vector functions whose  derivative (with respect to $x$, defined in the weak sense~\cite{evans}) is square-integrable in the $[0,1]$ interval. For simplicity we will simply write $L^2$ and $H^1$. If $f\in L^2$ or $f\in H^1$ its norm will be written as $\Vert f \Vert_{L^2}$ or $\Vert f \Vert_{H^1}$, respectively, and computed with the following expressions
\begin{equation}
\Vert f \Vert_{L^2}=\int_0^1 \vert f(x) \vert^2 dx, \,
\Vert f \Vert_{H^1}=\Vert f \Vert_{L^2}+\int_0^1 \left| \frac{\partial f(x)}{\partial x} \right|^2 dx,
\end{equation}
where $\vert\cdot \vert$ denotes the regular Euclidean norm. In addition we will use $L^2$ spaces with respect to time, which are analogously defined. Rather than using a more complex notation, we will denote the $L^2$ norm with respect to time equally as $\Vert \cdot \Vert_{L^2}$, and since it will only be used for functions only depending on time it should be clear from the context what $L^2$ norm we are referring to.

Define $C$ as a diagonal matrix of constant positive coefficients, i.e.,
\begin{equation}
C=\left[\begin{array}{cccc}c_1&0&\hdots&0 \\ 0&c_2&\hdots&0 \\
\vdots&\vdots&\ddots&\vdots\\
 0&0&\hdots & c_n\end{array}\right],
 \end{equation}
 with $c_1,c_2,\hdots,c_n>0$, whose values can be chosen but should be sufficiently large (see Section~\ref{sec-stabtarget}).

Next, we state our main result that solves the stabilization problem in $H^1$. 
\begin{thm}\label{th-main}
Consider system (\ref{eqn-reactdif}), (\ref{eqn-u0})--(\ref{eqn-u1}) with initial condition $u_0\in H^1$ and feedback control law
\begin{equation}
U(t)=\int_0^1 K(1,\xi) u(\xi,t) d\xi+b(t),
\end{equation}
where the kernel matrix $K(x,\xi)$ is a solution from the following hyperbolic \emph{matrix} system of PDEs
\begin{eqnarray}
&&\partial_x\left(\Sigma(x) K_{x}\right) -\partial_{\xi}\left( K_{\xi}\Sigma(\xi)\right)+ \Phi(x)  K_x
+ K_{\xi}\Phi(\xi) 
\nonumber \\Ê&=&K(x,\xi)\Lambda(\xi)+CK(x,\xi)+K(x,\xi) \Phi'(\xi), \label{eqn-Kdif}
\end{eqnarray}
in the domain $\mathcal{T}=\{(x,\xi):0\leq \xi \leq x \leq 1\}$,
with boundary conditions
\begin{eqnarray}
&&\Phi(x) K(x,x)-K(x,x)\Phi(x)+\Lambda(x)+C+K_{\xi}(x,x)\Sigma(x)
\nonumber \\ &&
+\Sigma(x) K_x(x,x)+\frac{d}{dx} \left(\Sigma(x)K(x,x)\right)=0,\label{eqn-Kbc1}\\
&&\Sigma(x) K(x,x)-K(x,x) \Sigma(x) =0,\label{eqn-Kbc2}\\
&&K_{ij}(x,0)=0,\quad i\leq j\label{eqn-Kbc3} 
\end{eqnarray}
and $b(t)$ is defined as
\begin{equation}\label{eqn-defb}
b(t)=\left(u_0(1,t)-\int_0^1 K(1,\xi) u_0(\xi) d\xi \right)\mathrm{e}^{-\alpha_1 t},
\end{equation}
for any chosen $\alpha_1>0$. Then, there is a unique $u(\cdot,t)\in H^1$ solution to  (\ref{eqn-reactdif}), (\ref{eqn-u0})--(\ref{eqn-u1}), and in addition, there exists a number $c^*$ depending only on the coefficients $\Sigma(x)$ and $\Phi(x)$, so that if the values of the coefficients of $C$ verify $c_i\geq c^*+\delta$, for all $i=1,\hdots,n$, and for some $\delta>0$, then the origin $u \equiv 0$ is exponentially stable in the $H^1$ norm, i.e.,
\begin{equation}
\Vert u(\cdot,t) \Vert_{H^1} \leq C_1 \mathrm{e}^{-C_2 t} \Vert u_0 \Vert_{H^1},
\end{equation}
with $C_1,C_2>0$, where $C_2=\min\{\alpha_1,2\delta\}$.
\end{thm}
\vspace{0.3cm}
In Theorem~\ref{th-main}, the main question is if the kernel equations (\ref{eqn-Kdif})--(\ref{eqn-Kbc3}) do indeed have a solution, as implicitly assumed in the theorem's statement. The next result answers this question.
\begin{thm}\label{th-kernels}
The kernel equations (\ref{eqn-Kdif})--(\ref{eqn-Kbc3}) possess a piecewise differentiable solution in the domain $\mathcal{T}$. In addition, the transformation defined by 
\begin{equation}
g(x)=f(x)-\int_0^x K(x,\xi) f(\xi)d\xi,
\end{equation}
is an invertible transformation. Both the transformation and its inverse map $H^1$ functions into $H^1$ functions, verifying
$$ \Vert g \Vert_{H^1} \leq K_1 \Vert f \Vert_{H^1},\,\Vert f \Vert_{H^1} \leq K_2 \Vert g \Vert_{H^1}.$$
\end{thm}

In the next sections we prove Theorem~\ref{th-main} and~\ref{th-kernels}, respectively in sections~\ref{sec-design} and~\ref{sec-kernels}.

\section{Control law design and closed-loop stability (proof of Theorem~\ref{th-main})}\label{sec-design} 

We stabilize (\ref{eqn-reactdif}), (\ref{eqn-u0})--(\ref{eqn-u1}) by applying the backstepping method. Next we explain the method and show that the origin of the resulting closed-loop system is stable in the $H^1$ norm.

\subsection{Backstepping transformation and target system}
The main idea of backstepping is to use a transformation mapping (\ref{eqn-reactdif}), (\ref{eqn-u0})--(\ref{eqn-u1}) into an stable \emph{target} system, which has to be adequately chosen. We select the following system
\begin{equation}\label{eqn-reacdiftarg}
w_t=\partial_x\left(\Sigma(x) w_{x}\right)+\Phi(x) w_x-Cw-G(x)w_x(0,t),
\end{equation}
where the target state $w$ is defined as
\begin{equation}
w=\left[\begin{array}{c}w_1\\ w_2 \\ \vdots \\ w_n \end{array}\right].
\end{equation}
with boundary conditions
\begin{equation}
w(0,t)=0,\qquad w(1,t)=b(t),\label{eqn-bctarg}
\end{equation}
whose stability properties will be studied in Section~\ref{sec-stabtarget}. The matrix $G(x)$ appearing in (\ref{eqn-reacdiftarg}) is a lower triangular matrix with zero diagonal, i.e.,
 \begin{equation}\label{eqn-G}
 G=\left[\begin{array}{ccccc}0&0&\hdots&0&0 \\ g_{21}(x)&0&\hdots&0&0 \\
\vdots&\vdots&\ddots&\vdots&\vdots\\
 g_{(n-1)1}(x)&g_{(n-1)2}(x)&\hdots &0& 0 \\
 g_{n1}(x)&g_{n2}(x)&\hdots &g_{n(n-1)}(x)& 0 \end{array}\right].
 \end{equation}
 The values of the non-zero entries of $G(x)$ are not arbitrary and will be set later in the design.
 
The backstepping transformation that maps $u$ into $w$ is defined as
\begin{eqnarray}\label{eqn-tran}
w(x,t)=u(x,t)-\int_0^x K(x,\xi) u(\xi,t) d\xi,
\end{eqnarray}
where the kernel matrix $K(x,\xi)$ is given by
 \begin{equation}
 K(x,\xi)=\left[\begin{array}{cccc}K^{11}(x,\xi)&K^{12}(x,\xi)&\hdots&K^{1n}(x,\xi) \\ K^{21}(x,\xi)&K^{22}(x,\xi)&\hdots&K^{2n}(x,\xi) \\
\vdots&\vdots&\ddots&\vdots\\
 K^{n1}(x,\xi)&K^{n2}(x,\xi)&\hdots & K^{nn}(x,\xi) \end{array}\right].
\end{equation}
Next we explain how to find conditions for $K(x,\xi)$ so that in fact (\ref{eqn-reacdiftarg}) holds.

\subsection{Finding the kernel equations}
First we establish   (\ref{eqn-Kdif})--(\ref{eqn-Kbc3}).
To find the equations that the kernel matrix $K(x,\xi)$ must verify, we take time and space derivatives in (\ref{eqn-tran})
\begin{eqnarray}\label{eqn-trandert}
w_t\hspace{-4pt}
&\hspace{-4pt}=\hspace{-4pt}&\hspace{-4pt} u_t-\int_0^x K u_t(\xi,t) d\xi,\\ \label{eqn-tranderx}
w_x\hspace{-4pt}
&\hspace{-4pt}=\hspace{-4pt}&\hspace{-4pt}u_x-\int_0^x K_x u(\xi,t) d\xi-K(x,x)u,\\ \label{eqn-tranderxx}
\partial_x\left(\Sigma(x) w_x\right)\hspace{-4pt}
&\hspace{-4pt}=\hspace{-4pt}&\hspace{-4pt}\partial_x\left(\Sigma(x) u_x\right)-\int_0^x \partial_x\left(\Sigma(x) K_x\right) u(\xi,t) d\xi
\nonumber \\ &&
-\partial_x\left(\Sigma(x) K(x,x) u(x,t) \right)
\nonumber \\ &&
-\Sigma(x) K_x(x,x) u(x,t).
\end{eqnarray}
and substituting (\ref{eqn-reactdif}) and (\ref{eqn-reacdiftarg})
 inside (\ref{eqn-trandert}) we find
 \begin{eqnarray}\label{eqn-trandert2}
&& \partial_x\left(\Sigma(x) w_{x}\right)+\Phi(x) w_x-Cw-G(x)w_x(0,t) \nonumber \\
&=&\partial_x\left(\Sigma(x) u_{x}\right)+\Phi(x) u_x+\Lambda(x)u
\nonumber \\Ê&&
-\int_0^x K(x,\xi) \left[
\partial_{\xi} \left(\Sigma(\xi) u_{\xi}(\xi,t) \right)+\Phi(\xi) u_{\xi}(\xi,t) 
\right. \nonumber \\Ê&& \left.
+\Lambda(\xi)u(\xi,t) 
\right]d\xi.
\end{eqnarray}
Using now (\ref{eqn-tranderx}) and (\ref{eqn-tranderxx}) we find
\begin{eqnarray}\label{eqn-trandert3}
&& -\int_0^x \partial_x\left(\Sigma(x) K_x\right) u(\xi,t) d\xi-\Sigma(x) K_x(x,x) u(x,t)
\nonumber \\ &&
-\partial_x\left(\Sigma(x) K(x,x) u(x,t) \right)
-\int_0^x \Phi(x)  K_x u(\xi,t) d\xi
\nonumber \\ &&
-\Phi(x) K(x,x)u
-Cu(x,t)+\int_0^x CK u(\xi,t) d\xi
\nonumber \\ &&
-G(x)u_x(0,t)
 \nonumber \\
&=&\Lambda(x)u
-\int_0^x K \left[
\partial_{\xi} \left(\Sigma(\xi) u_{\xi}(\xi,t) \right)+\Phi(\xi) u_{\xi}(\xi,t) 
\right. \nonumber \\Ê&& \left.
+\Lambda(\xi)u(\xi,t) 
\right]d\xi,
\end{eqnarray}
and integrating by parts twice in the right-hand side integral of (\ref{eqn-trandert3}) we find
\begin{eqnarray}\label{eqn-trandert4}
\hspace{-4pt}
&\hspace{-4pt}\hspace{-4pt}&\hspace{-4pt}  -\int_0^x \partial_x\left(\Sigma(x) K_x\right)u(\xi,t) d\xi-\Sigma(x) K_x(x,x) u(x,t)
\nonumber \\ &&
-\partial_x\left(\Sigma(x) K(x,x) u(x,t) \right)
-\int_0^x \Phi(x)  K_x u(\xi,t) d\xi
\nonumber \\ &&
-\Phi(x) K(x,x)u
-Cu(x,t)+\int_0^x CK u(\xi,t) d\xi
\nonumber \\ &&
-G(x)u_x(0,t)
\nonumber \\
\hspace{-4pt}
&\hspace{-4pt}\leq\hspace{-4pt}&\hspace{-4pt} \Lambda(x)u
-K(x,x) \Sigma(x) u_{x}(x,t) 
+ K(x,0) \Sigma(0) u_{x}(0,t) 
\nonumber \\ &&
+ K_{\xi}(x,x) \Sigma(x) u(x,t) 
-\int_0^x \partial_{\xi} \left(K_{\xi} \Sigma(\xi)\right) u(\xi,t) d\xi
\nonumber \\ &&
+\int_0^x K_{\xi}\Phi(\xi) u(\xi,t) d\xi+\int_0^x K\Phi'(\xi) u(\xi,t) d\xi
\nonumber \\ &&
- K(x,x)\Phi(x) u(x,t) 
-\int_0^x K\Lambda(\xi)u(\xi,t) d\xi,
\end{eqnarray}
where the boundary condition of $u$ at $x=0$ has been used. We separately collect the terms in (\ref{eqn-trandert4}) affecting $u(x,t)$, $u_x(x,t)$, $u_x(0,t)$ and in the integrals, reaching four equations that need to be independently verified if (\ref{eqn-trandert4}) is to hold for any value of $u$. These equations are as follows. First we find a hyperbolic matrix PDE
\begin{eqnarray}
&&  \partial_x\left(\Sigma(x) K_x\right)
- \partial_{\xi} \left(K_{\xi} \Sigma(\xi)\right) 
+ \Phi(x)  K_x
+ K_{\xi}\Phi(\xi) \nonumber \\
&=&
K\Lambda(\xi)+CK
- K\Phi'(\xi), \label{eqn-kernel1}
\end{eqnarray}
where we have omitted the dependence of $K(x,\xi)$. Next, we find three additional conditions
\begin{eqnarray} \label{eqn-Gx}
G(x)&=&- K(x,0) \Sigma(0) ,\\  \label{eqn-Kbc1}
K(x,x) \Sigma(x)&=&\Sigma(x) K(x,x) ,\\
C+\Lambda(x)
&=&-\Sigma(x) K_x(x,x)-\partial_x\left(\Sigma(x) K(x,x) \right)
\nonumber \\Ê&&
 -\Phi(x) K(x,x)- K_{\xi}(x,x) \Sigma(x) 
 \nonumber \\Ê&&
+K(x,x)\Phi(x). \label{eqn-Kbc2}
\end{eqnarray}

Finally, using the structure of $G$ given in (\ref{eqn-G}) in the boundary condition (\ref{eqn-Kbc1}), we find that, on the one hand,
$$
K_{ij}(x,0)=0,\qquad \forall j\geq i,$$
which is the boundary condition explicitly named in (\ref{eqn-Kbc3}),
and on the other hand, 
$$
g_{ij}(x)=- K_{ij}(x,0) \epsilon_j(0),\qquad \forall j<i,
$$
which is the \emph{definition} of the non-zero coefficients of $G(x)$.

\subsection{Target system stability}\label{sec-stabtarget}
The following result holds for the target system.
\begin{prop} \label{prop-wptarget}
The system (\ref{eqn-reacdiftarg}) with boundary conditions (\ref{eqn-bctarg}) and initial conditions $w_0\in H^1$, verifying $w_0(1)=b(0)$ and with $b,\dot b \in L^2((0,\infty])$ has an unique solution $w(\cdot,t) \in H^1$. In addition, there exists a number $c^*$ depending only on the coefficients of the system so that if the coefficients of $C$ verify $c_i\geq c^*+\delta$, for all $i=1,\hdots,n$, and for some $\delta>0$,  then the origin $w \equiv 0$ is exponentially stable in the $H^1$ norm, i.e.,
\begin{equation}
\Vert w(\cdot,t) \Vert_{H^1} \leq D_1 \mathrm{e}^{-2\delta t} \Vert w_0 \Vert_{H^1}+D_3 \left(\Vert b \Vert_{L^2}+\Vert \dot b \Vert_{L^2} \right).
\end{equation}
\end{prop}
\begin{proof}
First we show the well-posedness result, which is standard, by noticing that the right-hand side of (\ref{eqn-reacdiftarg}) defines a parabolic operator, the 0-th order compatibility conditions are verified (due to the fact that $w_0(1)=b(0)$), and $w_0 \in H^1$. Thus $w\in L^2(0,T;H^2([0,1]))$ and $w_t \in L^2(0,T;L^2([0,1]))$, see any standard PDE textbook such as e.g.~\cite[p.382]{evans}. The trace terms $w_x(1,t)$ are not classically considered but do not complicate the proof as they are as regular as the differential operator.
Now, to show the stability result, consider the following Lyapunov functionals
\begin{eqnarray}
V_1(t)&=&\frac{1}{2} \int_0^1  w^T Q w dx\label{eqn-V1},\\
V_2(t)&=&\frac{1}{2} \int_0^1 w_x^T Q w_x dx,\label{eqn-V2},\\
V_3(t)&=&\frac{1}{2} \int_0^1 w_{xx}^T Q w_{xx} dx,\label{eqn-V3}
\end{eqnarray}
where the space and time dependence of $w$ has been omitted for simplicity. The matrix $Q$ is a square diagonal matrix, with diagonal elements denoted  as $q_1,\hdots,q_n$ and chosen positive, so that $\underline{q}\leq q_i \leq \bar q$, so that $Q>0$. It is obvious that $V_1+ V_2$, is equivalent to the $H^1$ norm of $u$, i.e., $K_3 (V_1+V_2) \leq \Vert u(x,\cdot) \Vert_{H^1} \leq K_4 (V_1+  V_2)$ for $K_3,K_4>0$.

Taking derivatives we obtain, for $\dot V_1$,
\begin{eqnarray}
\dot V_1&=& \int_0^1 w^TQ \left(
\partial_x\left(\Sigma(x) w_{x}\right) \right)dx
\nonumber \\Ê&&
+
 \int_0^1 w^T Q \left(\Phi(x) w_x-Cw-G(x)w_x(0,t)
 \right) dx
 \nonumber \\Ê&=&
- \int_0^1 w_x^T Q\Sigma(x) w_{x}dx
+b^T(t)  Q \Sigma(1) w_{x}(1,t) 
\nonumber \\Ê&&
+
 \int_0^1 w^T Q \Phi(x) w_x dx
 - \int_0^1 w^T QCw dx\nonumber \\Ê&&
 -\left(\int_0^1 w^T QG(x)dx\right) w_x(0,t)
\end{eqnarray}
Now, assuming that, for all $x\in[0,1]$, the coefficients verify the following bounds: $\Vert \Phi(x) \Vert \leq p$, $\underline{c}\leq c_i \leq \bar c $, $\vert g_{ij}(x)\vert \leq g$, where $\Vert \cdot \Vert$ is the matrix operator 2-norm. Then
\begin{eqnarray}
\dot V_1&\leq& 
-2\underline{\epsilon}V_2
+ \bar{\epsilon}\bar{q} \left| b(t)^T w_{x}(1,t)\right|+2p(V_1+V_2)\nonumber \\Ê&&
-(2\underline c-1) V_1
+\frac{g^2}{2}\vert \sqrt{Q}L w_x(0,t) \vert^2,
\end{eqnarray}
where $L$ is a lower triangular matrix with zero diagonal and unity coefficients, i.e.,
 \begin{equation}\label{eqn-L}
 L=\left[\begin{array}{ccccc}0&0&\hdots&0&0 \\ 1&0&\hdots&0&0 \\
\vdots&\vdots&\ddots&\vdots&\vdots\\
1&1&\hdots &0& 0 \\
1&1&\hdots &1& 0 \end{array}\right].
 \end{equation}

Similarly, taking derivative in $V_2$ we obtain
\begin{eqnarray}
\dot V_2&=&
 \int_0^1 w_x^TQ w_{xt} dx
\nonumber \\Ê&=& 
 -\int_0^1 w_{xx}^T Q w_{t} dx
 +\left. w_x^T Qw_{t}\right|_0^1
\nonumber \\Ê&=&  
 -\int_0^1 w_{xx}^TQ \left(
\partial_x\left(\Sigma(x) w_{x}\right) \right)dx
 + w_x(1,t)^TQ \dot b(t)
\nonumber \\Ê&&
-
 \int_0^1 w_{xx}^T Q \left(\Phi(x) w_x-Cw-G(x)w_x(0,t)
 \right) dx
 \nonumber \\Ê&=&
- \int_0^1 w_{xx}^T Q\Sigma(x) w_{xx}dx
- \int_0^1 w_{xx}^TQ \Sigma'(x) w_{x}dx
\nonumber \\Ê&&
 + w_x(1,t)^T Q \dot b(t)+
 \int_0^1 w_{xx}^TQ \Phi(x) w_x dx
\nonumber \\Ê&&
 - \int_0^1 w_{x}^T QCw_x dx
 + w_x(1,t)^TQ C b(t)\nonumber \\Ê&&
 -\left(\int_0^1 w_{xx}^T Q G(x)dx\right) w_x(0,t).
\end{eqnarray}
Therefore, defining  $\epsilon'_i(x) \leq \bar{\epsilon}'$ and using the previously defined bounds,
\begin{eqnarray}
\dot V_2&\leq& 
-\left[\underline{\epsilon}-\left(\frac{\alpha_2\bar{\epsilon}'}{2}+\frac{p\alpha_3}{2}+\frac{g\alpha_4}{2} \right)\right] V_3-2\underline c V_2
\nonumber \\Ê&&
+\left(\frac{\bar{\epsilon}'}{\alpha_2}+\frac{p}{\alpha_3}\right)  V_2
+\bar q \left| \left(\dot b(t)+Cb(t)\right)^T w_{x}(1,t)\right|
\nonumber \\Ê&&
+\frac{g}{2\alpha_4}\vert \sqrt{Q}L w_x(0,t) \vert^2),
\end{eqnarray}
for $\alpha_2,\alpha_3,\alpha_4>0$. Now we have the following inequality
\begin{equation}
\left|w_x(1,t)\right|^2\leq
\frac{2}{\underline q} (V_2+V_3)
\end{equation}
which is proven by considering that
\begin{equation}
w_x(1,t)=\int_0^1 \left[(x-1)w_{x}(x,t)\right]_x dx,
\end{equation}
therefore
\begin{equation}
\left|w_x(1,t)\right|\leq
\int_0^1 \left[\vert w_{x}(x,t)\vert+\vert w_{xx}(x,t)\vert\right] dx,
\end{equation}
which squared, gives the inequality. In a similar fashion, we can prove that
\begin{equation}
\left|\sqrt{Q}Lw_x(0,t)\right|\leq
\int_0^1 \left[\vert \sqrt{Q}Lw_{x}\vert+\vert \sqrt{Q}Lw_{xx}\vert\right] dx,
\end{equation}
thus
\begin{equation}
\left|\sqrt{Q}Lw_x(0,t)\right|^2\leq 
\int_0^1 \frac{w_x^TL^TQLw_{x}+ w_{xx}^TL^TQLw_{xx}}{2} dx,
\end{equation}

Considering now $V=V_1+ V_2$, we obtain
\begin{eqnarray}
\dot V&\leq&
-V_1 \left[2\underline c-\left(2p+1 \right)\right]
\nonumber \\Ê&&
-V_2 \left[2\underline{\epsilon}+2\underline c-\left(2p+ \left(\frac{\bar{\epsilon}'}{\alpha_2}+\frac{p}{\alpha_3}\right)\right)\right]
\nonumber \\Ê&&
-V_3 \left[\underline{\epsilon}-\left(\frac{\alpha_2\bar{\epsilon}'}{2}+\frac{p\alpha_3}{2}+\frac{g\alpha_4}{2} \right)\right]
\nonumber \\Ê&&
+\frac{g}{2} \left(\frac{1}{\alpha_4}+g\right)\int_0^1 \frac{w_x^TL^TQLw_{x}+ w_{xx}^TL^TQLw_{xx}}{2} dx
\nonumber \\Ê&&
+\frac{\bar q}{2} \left((1+\bar c)+\bar \epsilon\right) \left( \vert\dot b(t)\vert+\vert b(t)\vert\right)\sqrt{\frac{2}{\underline q} (V_2+V_3)},
\end{eqnarray}
Choose now $\alpha_2=\frac{\underline{\epsilon}}{3\bar \epsilon'}$,$\alpha_3=\frac{\underline{\epsilon}}{3\bar p}$,$\alpha_4=\frac{\underline{\epsilon}}{3g}$ so that $\left(\frac{\alpha_2\bar{\epsilon}'}{2}+\frac{p\alpha_3}{2}+\frac{g\alpha_4}{2} \right)<\underline{\epsilon}/2$. Call $K_5=2p+1$, $K_6=\left(2p+\frac{\underline{\epsilon}}{4} + \frac{3}{\underline{\epsilon}} \left( \bar{\epsilon}'^2+p^2\right)\right)-2\underline{\epsilon}$, $K_7=\frac{\bar q^2}{2\underline{\epsilon}\underline q} \left((1+\bar c)+\bar \epsilon\right)^2$, $K_8=\frac{g^2}{2} \left(\frac{1}{3 \underline{\epsilon}}+1\right)$. Then:
\begin{eqnarray}
\dot V&\leq&
-V_1 \left[2\underline c-K_5\right]
-V_2 \left[2\underline c-K_6\right]
-\frac{\underline{\epsilon}}{4}V_3 
\nonumber \\Ê&&
+K_7\left( \vert\dot b(t)\vert^2+\vert b(t)\vert^2\right)
\nonumber \\Ê&&
+K_8\int_0^1 \frac{w_x^TL^TQLw_{x}+ w_{xx}^TL^TQLw_{xx}}{2} dx,
\end{eqnarray}
and defining $c^*=\frac{1}{2} \max\{K_5,K_6+\frac{\underline{\epsilon}}{4}\}$ (which only depends on the bounds of $\Sigma(x)$ and $\Phi(x)$), we get that if $\underline c\geq c^*+\delta$, we obtain
\begin{eqnarray}
\dot V&\leq&-2 \delta V 
+K_7\left( \vert\dot b(t)\vert^2+\vert b(t)\vert^2\right)
\nonumber \\Ê&&
-\int_0^1 \frac{w_x^TRw_{x}+ w_{xx}^TRw_{xx}}{2} dx,
\end{eqnarray}
where  $R=\frac{\underline{\epsilon}}{4}Q-K_8L^TQL$ and $D_2$ can be set as large as desired. Assume for the moment that $R$ is definite positive. Then, applying Gronwall's inequality, we obtain
\begin{eqnarray}
V &\leq V(0) &\mathrm{e}^{-2\delta t}+K_7 \int_0^t \mathrm{e}^{-2\delta(t-\tau)} 
\left( \vert\dot b(\tau)\vert^2+\vert b(\tau)\vert^2\right) d\tau \nonumber \\ 
&\leq&V(0) \mathrm{e}^{-2\delta t} + K_7\left(\Vert b \Vert_{L^2}+\Vert \dot b \Vert_{L^2} \right),
\end{eqnarray}
and then the proposition is proved. It only remains to prove that $R$ can be made a positive definite matrix by adequately choosing the coefficients of $Q$. To see if this is possible, let us check what is $L^TQL$. First notice that $(QL)_{ij}=q_i$ if $j<i$ and zero otherwise. Then
 \begin{equation}\label{eqn-LQL}
 (L^TQL)_{ij}=\sum_{l=1}^n L_{li}  (QL)_{lj}
 =\sum_{l=i+1}^n (QL)_{lj}=\hspace{-3pt}
 \sum_{\max\{i,j\}+1}^{n}\hspace{-3pt} q_l,
  \end{equation}
  where the sum is considered to be zero if $i=n$ and/or $j=n$.
  Let us now prove by induction on the dimension $n$ that $R(Q)=\frac{\underline{\epsilon}}{4}Q-K_8L^TQL$ can always be made positive definite. Call $Q_n$ the matrix that we will find for each dimension, and $M_n=L_n^TQ_nL_n$. For $n=1$, since $Q_1=q_1>0$ and $M_1=0$, the result is obvious and $q_1$ can be chosen arbitrarily. For $n>1$, we can construct both $Q_n$ and $M_n$ from the previous $Q_{n-1}$ and $M_{n-1}$ as follows
   \begin{equation}
 Q_n=\left[\begin{array}{cc}Q_{n-1}&0\\
0&q_n\end{array}\right],\quad
 M_n=\left[\begin{array}{cc}M_{n-1}+q_n\mathbf{J}_{n-1}&0\\
0&0\end{array}\right],
 \end{equation}
 where $\mathbf{J}_{n-1}$ is a square matrix of dimension $n-1$ full of ones. Assume now that $R(Q_{n-1})$ is positive definite. In particular this means that all the eigenvalues of $R(Q_{n-1})$ are positive, and since $R$ is symmetric, they are also real. Call $\lambda_{\mathrm{min}}$ the smallest eigenvalue of a square matrix. Denote  $\mu_{n-1}=\lambda_{\mathrm{min}}(R(Q_{n-1}))>0$. Choosing $q_n=\frac{\mu_{n-1}}{2K_8(n-1)}$, we obtain
 \begin{equation}
 R(Q_n)=
\left[\begin{array}{cc}R(Q_{n-1})-\frac{\mu_{n-1}}{2(n-1)}\mathbf{J}_{n-1}&0\\
0&\frac{\underline{\epsilon} \mu_{n-1}}{8K_8(n-1)}\end{array}\right],\quad
 \end{equation}
 and computing the eigenvalues of $R(Q_n)$, we obtain one eigenvalue equal to $\frac{\underline{\epsilon} \mu_{n-1}}{8K_8(n-1)}>0$, plus the eigenvalues of $R(Q_{n-1})-\frac{\mu_{n-1}}{2(n-1)}\mathbf{J}_{n-1}$. Now, from Weyl's inequality~\cite[p. 239]{johnson} we then have that
 \begin{eqnarray}
&& \lambda_{\mathrm{min}}\left(R(Q_{n-1})-\frac{\mu_{n-1}}{2(n-1)}\mathbf{J}_{n-1}\right)\nonumber \\
 &\leq& \lambda_{\mathrm{min}}\left(R(Q_{n-1})\right)-\lambda_{\mathrm{min}}\left(\frac{\mu_{n-1}}{2(n-1)}\mathbf{J}_{n-1}\right)\nonumber \\
 &=&\frac{\mu_{n-1}}{2}>0,
\end{eqnarray}
where we have used that the eigenvalues of $\mathbf{J}_{n-1}$ are $0$ (repeated $n-2$ times) and $n-1$~\cite[p. 65]{johnson}. Therefore the newly formed $R(Q_n)>0$, and the proposition is proved.
\end{proof}

\subsection{Proof of Theorem~\ref{th-main}}\label{sec-proofmain}
Assume for the moment that Theorem~\ref{th-kernels} holds and there is a solution to the kernel equations such that the transformation (\ref{eqn-tran}) is invertible and both the transformation and its inverse map $H^1$ functions into $H^1$ functions. Consider now the target system equation (\ref{eqn-reacdiftarg}) with boundary conditions (\ref{eqn-bctarg}) and initial conditions $w_0(x)$ given by applying the backstepping transformation (\ref{eqn-tran}) to the initial conditions of $u$, $u_0(x)$, i.e.,
\begin{equation}\label{eqn-w0}
w_0(x)=u_0(x)-\int_0^x K(x,\xi) u_0(\xi)d\xi.
\end{equation}
Then, since $u_0 \in H^2$, we have $w_0$ in $H^2$. In addition, given the definition of $b(t)$ (see Equation~\ref{eqn-defb}), we have $w_0(1)=b(0)$, and $b,\dot b \in L^2([0,\infty)]$. Thus the conditions for well-posedness of Proposition~\ref{prop-wptarget} are fulfilled and we obtain well-posedness for $u$ in $H^1$ given the properties of the transformation. In addition it is obvious that $\Vert b \Vert_{L^2}+\Vert \dot b \Vert_{L^2} \leq D_4 \mathrm{e}^{-\alpha_1 t} \Vert u_0 \Vert_{H^1}$. We then obtain, if $c_i\geq c^*+\delta$,
\begin{eqnarray}
\Vert u(\cdot,t) \Vert_{H^1} \hspace{-4pt}
&\hspace{-4pt}\leq\hspace{-4pt}&\hspace{-4pt}  K_2 \Vert w(\cdot,t) \Vert_{H^1} \nonumber \\
\hspace{-4pt}
&\hspace{-4pt}\leq\hspace{-4pt}&\hspace{-4pt} 
K_2 \left(D_1 \mathrm{e}^{-2\delta t} \Vert w_0 \Vert_{H^1}+D_3 \left(\Vert b \Vert_{L^2}+\Vert \dot b \Vert_{L^2} \right) \right) 
\nonumber \\
\hspace{-4pt}
&\hspace{-4pt}\leq\hspace{-4pt}&\hspace{-4pt}  K_2  (K_1D_1+D_4)\mathrm{e}^{-D_5 t}\Vert u_0 \Vert_{H^1},
\end{eqnarray}
where $D_5=\min\{\alpha_1,2\delta \}$. Thus Theorem~\ref{th-main} is proved. \qed
\section{Well-posedness of the kernel equations (proof of Theorem~\ref{th-kernels})}\label{sec-kernels}
To prove Theorem~\ref{th-kernels}, we are going to write the kernel equations (\ref{eqn-Kdif})--(\ref{eqn-Kbc3}) in a different form. Then, we can use Theorem A.1 of~\cite{long-nonlinear}.

Define first
\begin{eqnarray}
 L(x,\xi)&=&\sqrt{\Sigma}(x)K_x(x,\xi)+K_{\xi}(x,\xi)\sqrt{\Sigma}(\xi)
\nonumber \\ 
&&+
F_1(x,\xi)K(x,\xi)+K(x,\xi) F_2(x,\xi) \label{eqn-Ldef}
\end{eqnarray}
where the functions $F_1$ and $F_2$ are to be found.

Now, we compute $ \sqrt{\Sigma}(x)L_x-L_{\xi}\sqrt{\Sigma}(\xi)$ using (\ref{eqn-Ldef}). It is worth noticing that the cross-derivatives of $K$ cancel out and the differential operator of  (\ref{eqn-Kdif}) appear. Replacing its value from (\ref{eqn-Kdif}), we obtain
\begin{eqnarray}
&& \sqrt{\Sigma}(x)L_x-L_{\xi}\sqrt{\Sigma}(\xi)
\nonumber \\ 
&=&
K(\Lambda(\xi)-\Phi'(\xi)-F_{2\xi}\sqrt{\Sigma}(\xi)-F_2^2)
\nonumber \\ &&
+(C+\sqrt{\Sigma}(x)
F_{1x}+F_1^2)K
\nonumber \\ &&
-
\left(\frac{\Sigma'(x)}{2}+\Phi(x)
-\sqrt{\Sigma}(x)F_1
-F_1\sqrt{\Sigma}(x)
   \right)K_x
   \nonumber \\ &&
+K_{\xi}\left(\frac{\Sigma'(\xi)}{2}-\Phi(\xi) 
-F_2 \sqrt{\Sigma}(\xi)
-\sqrt{\Sigma}(\xi)  F_2
\right)
\nonumber \\ &&
+\sqrt{\Sigma}(x)
K F_{2x}
-F_{1\xi}K\sqrt{\Sigma}(\xi)
\nonumber \\ &&
-F_1 L + L F_2\label{eqn-KLsub}
\end{eqnarray}
Now, $F_1$ and $F_2$ are chosen so that the second and third lines of (\ref{eqn-KLsub}) cancel out. This is always possible~\cite{jameson}, by defining
\begin{eqnarray}
(F_1)_{ij}&=&\frac{\delta_{ij}\frac{\epsilon_i'(x)}{2}+\phi_{ij}(x)}{\sqrt{\epsilon_i}(x)+\sqrt{\epsilon_j}(x)},\\
(F_2)_{ij}&=&\frac{\delta_{ij}\frac{\epsilon_i'(\xi)}{2}-\phi_{ij}(\xi)}{\sqrt{\epsilon_i}(\xi)+\sqrt{\epsilon_j}(\xi)},
\end{eqnarray}
and noticing that $F_1$ only depends on $x$ and $F_2$ only depends on $\xi$, the fourth line of (\ref{eqn-KLsub}) is also zero.
Thus our original $n\times n$ system  (\ref{eqn-Kdif}) is replaced by a $n^2\times n^2$ system of first-order hyperbolic equation on the same domain $\mathcal{T}$, namely
\begin{eqnarray}
\sqrt{\Sigma}(x)K_x+K_{\xi}\sqrt{\Sigma}(\xi)
&=& L-
F_1(x)K-KF_2(\xi),\quad \label{eqn-KL1} \\
 \sqrt{\Sigma}(x)L_x-L_{\xi}\sqrt{\Sigma}(\xi)
&=&KF_3(\xi)+F_4(x) K
\nonumber \\Ê&&
-F_1(x) L + L F_2(\xi),\label{eqn-KL2}
\end{eqnarray}
where $F_3(\xi)=\Lambda(\xi)-\Phi'(\xi)- F_{2\xi}\sqrt{\Sigma}(\xi)-F_2^2$, $F_4(x)=C+\sqrt{\Sigma}(x)
F_{1x}+F_1^2$, which are virtually identical to the kernel equations appearing in~\cite{long-nonlinear} and~\cite{long-linear} (there are some differences in the right-hand side coefficients, but they do not affect the proofs). It remains to be seen if the boundary conditions are the same.

To find the boundary conditions for $L$, we need to analyze separate cases depending on the position of each coefficient $L_{ij}$ and $K_{ij}$ in the kernel matrices $L$ and $K$. First, (\ref{eqn-Kbc2}), namely $K(x,x) \Sigma(x)=\Sigma(x) K(x,x)$ can be written as $K_{ij}(x,x) (\epsilon_i(x)-\epsilon_j(x))=0$. This condition is automatically verified if $i=j$, otherwise $K_{ij}(x,x)=0$. This allows us to write (\ref{eqn-Kbc1}) as
\begin{eqnarray}
0&=& \phi_{ij}(x) K_{jj}(x,x)-\phi_{ij}(x)K_{ii}(x,x)+\lambda_{ij}(x)+\delta_{ij}c_{i}
\nonumber \\ &&
+K_{ij\xi}(x,x)\epsilon_j(x)
+\epsilon_i(x) K_{ijx}(x,x)
\nonumber \\ &&
+\frac{d}{dx}\left(\epsilon_i(x) K_{ij}(x,x)\right),\label{eqn-Kbc1bis}
\end{eqnarray}
and similarly, we can solve for $L_{ij}(x,x)$ in (\ref{eqn-KL1}), finding
\begin{eqnarray}
 L_{ij}(x,x)&=&\sqrt{\epsilon_i}(x)K_{ijx}(x,x)+\sqrt{\epsilon_j}(\xi) K_{ij\xi}(x,x)\nonumber \\ &&
+
F_{ij1}(x)K_{jj}(x,x)+F_{ij2}(x)K_{ii}(x,x),\quad\label{eqn-Lij} \end{eqnarray}

If $i=j$, then (\ref{eqn-Kbc1bis}) reduces to
\begin{eqnarray}
0&=& \lambda_{ii}(x)+c_{i}
\nonumber \\ &&
+\epsilon_i'(x) K_{ii}(x,x)
+2\epsilon_i(x) \frac{d}{dx}\left( K_{ii}(x,x)\right),
\end{eqnarray}
which integrates (combined with (\ref{eqn-Kbc3})) to
\begin{equation}
K_{ii}=\frac{-1}{\sqrt{\epsilon(x)}}\int_0^x \dfrac{\lambda_{ii}(\xi)+c_{i}}{2\sqrt{\epsilon(\xi)}}d\xi
\end{equation}
In addition, (\ref{eqn-Lij}) reduces to
\begin{eqnarray}
 L_{ii}(x,x)&=&\sqrt{\epsilon_i}(x)\frac{d}{dx} K_{ii}(x,x)\nonumber \\ &&
+(F_{ii1}(x)+F_{ii2}(x))K_{ii}(x,x)
\nonumber \\ &=&
\sqrt{\epsilon_i}(x)\frac{d}{dx} K_{ii}(x,x)\nonumber \\ &&
+\frac{\epsilon_i'(x)}{2\sqrt{\epsilon_i}(x)} K_{ii}(x,x)
\nonumber \\ &=&-
\frac{ \lambda_{ii}(x)+c_{i}}{2\sqrt{\epsilon_i}(x)}
,\label{eqn-Lij2} \end{eqnarray}
If $i\neq j$, then since $K_{ij}(x,x)=0$, we get $K_{ijx}(x,x)=-K_{ij\xi}(x,x)$. Therefore we obtain, from  (\ref{eqn-Kbc1bis}),
\begin{eqnarray}
0&=&\lambda_{ij}(x)+ \phi_{ij}(x)\left( K_{jj}(x,x)-K_{ii}(x,x)\right)
\nonumber \\ &&
+ K_{ijx}(x,x) (\epsilon_i(x)-\epsilon_j(x))
\end{eqnarray}
and from (\ref{eqn-Lij}),
\begin{eqnarray}
 L_{ij}(x,x)&=&K_{ijx}(x,x) (\sqrt{\epsilon_i(x)}-\sqrt{\epsilon_j(x)})
 \nonumber \\ &&
 +
F_{ij1}(x)K_{jj}(x,x)+F_{ij2}(x)K_{ii},\label{eqn-Lijbis} \end{eqnarray}
which combined gives us
\begin{eqnarray}
 L_{ij}(x,x)&=&\frac{K_{ijx}(x,x) (\epsilon_i(x)-\epsilon_j(x))}{ \sqrt{\epsilon_i(x)}+\sqrt{\epsilon_j(x)}}
 \nonumber \\ &&
 +
F_{ij1}(x)K_{jj}(x,x)+F_{ij2}(x)K_{ii}
\nonumber \\Ê&=&
-\frac{\lambda_{ij}+\phi_{ij}(x)\left( K_{jj}(x,x)-K_{ii}(x,x)\right)}{ \sqrt{\epsilon_i(x)}+\sqrt{\epsilon_j(x)}}
 \nonumber \\ &&
+
F_{ij1}(x)K_{jj}(x,x)+F_{ij2}(x)K_{ii}
 \nonumber \\ &=&-\frac{\lambda_{ij}}{ \sqrt{\epsilon_i(x)}+\sqrt{\epsilon_j(x)}},
\end{eqnarray}
when introducing the definitions of $F_1$ and $F_2$. Thus we are finally led to the following combination of boundary conditions
\begin{itemize}
\item If $i=j$, then simply
\begin{eqnarray}
L_{ii}(x,x)&=&-\frac{\lambda_{ii}(x)+c_{i}}{2\sqrt{\epsilon_i(x)}},\\
K_{ii}(x,0)&=&0,
\end{eqnarray}
\item If $i<j$ then
\begin{eqnarray}
K_{ij}(x,x)&=&K_{ij}(x,0)=0,\\
L_{ij}(x,x)&=&-\frac{\lambda_{ij}(x)}{\sqrt{\epsilon_i(x)}+\sqrt{\epsilon_j(x)}},\\
\end{eqnarray}
\item Finally if $i>j$ and $\epsilon_i\neq\epsilon_j$ then
\begin{eqnarray}
K_{ij}(x,x)&=&0,\\
K_{ij}(1,\xi)&=&l_{ij}(\xi)\label{eqn-arbitrary},\\
L_{ij}(x,x)&=&-\frac{\lambda_{ij}(x)}{\sqrt{\epsilon_i(x)}+\sqrt{\epsilon_j(x)}},
\end{eqnarray}
and the additional condition $g_{ij}(x)=- K_{ij}(x,0) \epsilon_j(0)$.
\end{itemize}
It must be noticed that (\ref{eqn-arbitrary}) are additional arbitrary conditions that are introduced for the kernel equations to be well-posed. These functions $l_{ij}(\xi)$ cannot be arbitrary, but need to verify certain compatibility conditions in the corner $\xi=1$ for the kernels to be piecewise differentiable (see~\cite{long-nonlinear} for details).

Comparing these boundary conditions with those verified by the kernels in~\cite{long-nonlinear} and~\cite{long-linear}, we can see that they are exactly the same (it must be noted than in the second of these papers the nomenclature for $K$ and $L$ is the opposite). Thus the results in these papers apply, and we obtain a piecewise differentiable and invertible kernel which can be readily verified to transform $H^1$ functions into $H^1$ functions.\qed

\section{Conclusion}\label{sec-conclusions}
This paper presents an extension of the backstepping method to coupled parabolic systems  with advection terms and spatially-varying coefficients. The result is more general than a recently published extension that only considered constant-coefficient coupled reaction-diffusion systems.

Interestingly, the basis of the result is finding an equivalence between the kernel equations for this case and the kernel equations for general hyperbolic 1-D coupled systems, which have recently been  established to be well-posed and piecewise differentiable. Thus, this paper unveils a direct connection between backstepping controllers for parabolic and hyperbolic systems.

Future work includes considering Neumann or Robin boundary conditions, which leads to slightly different kernel equations, and observer design, which will allow to consider output-feedback controllers.

\ifCLASSOPTIONcaptionsoff
  \newpage
\fi

\bibliographystyle{IEEEtran}

\begin{IEEEbiography}[{\includegraphics[width=1in,height=1.25in,clip,keepaspectratio]{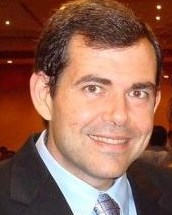}}] 
{Rafael Vazquez}(S'05-M'08-SM'15) received the M.S. and Ph.D. degrees in aerospace engineering from the University of California, San Diego (USA), and degrees in electrical  engineering and mathematics from the University of Seville (Spain). He is  an Associate Professor in the Aerospace Engineering and Fluid Mechanics Department at the University of Seville, where he is currently Chair of the Department. His research interests include control theory, distributed parameter systems, and optimization, with applications to flow control, ATM, UAVs, and orbital mechanics. He is coauthor of the book Control of Turbulent and Magnetohydrodynamic Channel Flows (Birkhauser, 2007). He currently serves as Associate Editor for Automatica.
\end{IEEEbiography}


\begin{IEEEbiography}[{\includegraphics[width=1in,height=1.25in,clip,keepaspectratio]{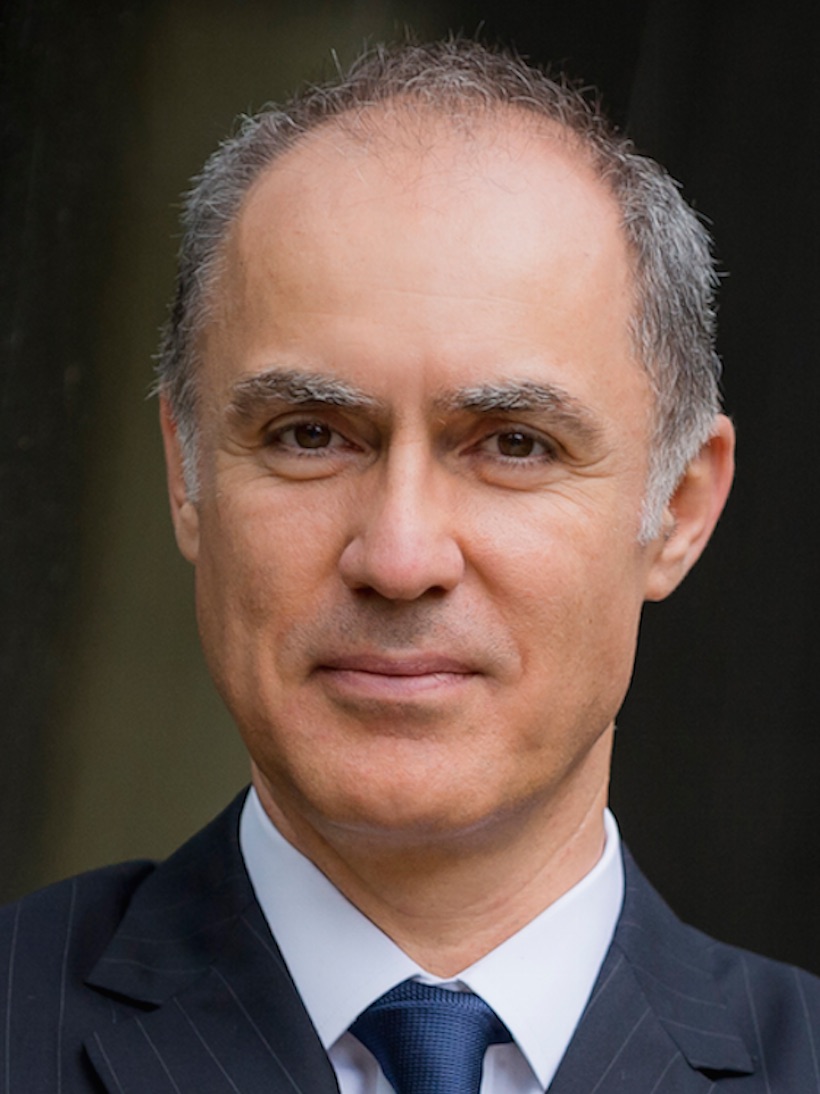}}] {Miroslav Krstic}
(S'92-M'95-SM'99-F'02) holds the Alspach endowed chair and is the founding director of the Cymer Center for Control Systems and Dynamics at UC San Diego. He also serves as Associate Vice Chancellor for Research at UCSD. As a graduate student, Krstic won the UC Santa Barbara best dissertation award and student best paper awards at CDC and ACC. Krstic is Fellow of IEEE, IFAC, ASME, SIAM, and IET (UK), Associate Fellow of AIAA, and foreign member of the Academy of Engineering of Serbia. He has received the PECASE, NSF Career, and ONR Young Investigator awards, the Axelby and Schuck paper prizes, the Chestnut textbook prize, the ASME Nyquist Lecture Prize, and the first UCSD Research Award given to an engineer. Krstic has also been awarded the Springer Visiting Professorship at UC Berkeley, the Distinguished Visiting Fellowship of the Royal Academy of Engineering, the Invitation Fellowship of the Japan Society for the Promotion of Science, and the Honorary Professorships from the Northeastern University (Shenyang) and the Chongqing University, China. He serves as Senior Editor in IEEE Transactions on Automatic Control and Automatica, as editor of two Springer book series, and has served as Vice President for Technical Activities of the IEEE Control Systems Society and as chair of the IEEE CSS Fellow Committee. Krstic has coauthored eleven books on adaptive, nonlinear, and stochastic control, extremum seeking, control of PDE systems including turbulent flows, and control of delay systems.
\end{IEEEbiography}

%
%
\end{document}